\documentclass[10pt]{amsart}
\usepackage{amsfonts}
\usepackage{amsmath}
\usepackage{graphicx,color}
\usepackage{amscd,amssymb,amsthm}
\usepackage[a4paper]{geometry}
\usepackage[colorlinks=true,pdfstartview={XYZ null null 1.00},pdfview=FitH,citecolor=blue,urlcolor=customgreen]{hyperref}
\numberwithin{equation}{section}
\newtheorem{lemma}{Lemma}
\newtheorem{theorem}{Theorem}
\newtheorem{corollary}{Corollary}
\theoremstyle{definition}
\newtheorem{remark}{Remark}
\definecolor{customgreen}{rgb}{0.0, 0.5, 0.0}

\title[Asymptotic expansions for the radii of starlikeness]{Asymptotic expansions for the radii\\ of starlikeness of normalised Bessel functions}

\author[\'A. Baricz]{\'Arp\'ad Baricz}
\address{Department of Economics, Babe\c{s}-Bolyai University, Cluj-Napoca, Romania}
\address{Institute of Applied Mathematics, \'Obuda University, Budapest, Hungary}
\email{bariczocsi@yahoo.com}

\author[G. Nemes]{Gerg\H{o} Nemes}
\address{Alfr\'ed R\'enyi Institute of Mathematics, Budapest, Hungary}
\email{nemes.gergo@renyi.hu}



\begin{document}

\begin{abstract} The asymptotic behaviour, with respect to the large order, of the radii of starlikeness of two types of normalised Bessel functions is considered. We derive complete asymptotic expansions for the radii of starlikeness and provide recurrence relations for the coefficients of these expansions. The proofs rely on the notion of Rayleigh sums and asymptotic inversion. The techniques employed in the paper could be useful to treat similar problems where inversion of asymptotic expansions is involved.
\end{abstract}

\maketitle

\section{Introduction}

Bessel functions are important members of classical special functions and appear frequently in problems related to mathematical analysis, mathematical physics or engineering. The study of their geometric properties from the point of view of complex function theory was initiated in the 1960's by Brown, Hayden, Kreyszig, Merkes, Scott, Robertson and Wilf (see for example the papers \cite{brown}, \cite{brown2}, \cite{brown3}, \cite{hayden}, \cite{todd}, \cite{merkes}, \cite{robertson} and \cite{wilf}). Following the proof by de Branges \cite{branges} of the famous Bieberbach conjecture on univalent functions in 1985, the hypergeometric functions came into the spotlight of complex function theory, and the geometric properties of these functions have been studied intensively (see, e.g, \cite{kustner}, \cite{kustner2}, \cite{miller}, \cite{samy} and \cite{rusch}). However, some of the results obtained on starlikeness and convexity of hypergeometric functions were stated in the form of sufficient conditions only (and not necessary), and in some cases it was not possible to find the optimal range of parameters so that the corresponding hypergeometric functions belong to some subclasses of univalent functions. Motivated by the results on hypergeometric functions, during the last decade, the geometric properties (such as univalence, starlikeness, convexity, uniform convexity, etc.) of Bessel functions have also been studied extensively (see for example the papers \cite{aby}, \cite{bks}, \cite{bos}, \cite{basz}, \cite{ds} and the references therein). The case of Bessel functions has been more fortunate: the radii and order of starlikeness and convexity (uniform convexity as well) for normalised Bessel functions of the first kind were determined successfully and studied intensively.

The main aim of the present paper is to make a contribution to the subject from the point of view of asymptotic analysis. By using the asymptotic behaviour of the Rayleigh sums (convergent Laurent series expansions at infinity) of the positive zeros of the Bessel function of the first kind, we derive complete asymptotic expansions, with respect to the large order, for the radii of starlikeness of two types of normalised Bessel functions. We provide recurrence relations, in terms of ordinary potential polynomials, for the coefficients of these expansions. The results presented in this paper improve, complement and extend the known results in the literature, and our method is more precise than that of Wilf's which is based on Euler--Rayleigh inequalities (see for example the papers \cite{wilf} and \cite{aby} for more details).

It is worth mentioning that our approach can be useful to tackle similar problems where asymptotic inversion is involved. For example, the derivations of the asymptotic expansions for the large zeros of Bessel functions \cite{olver} and for the inverse of the normalised incomplete gamma function \cite{nemes} can be supplemented by our methods (the asymptotic nature of the expansions in these works was not proved rigorously).

The rest of the paper is organised as follows. In Section \ref{Sec2}, we present the main results of the paper: the asymptotic expansions for the radii of starlikeness of two kinds of normalised Bessel functions. The proofs of the theorems and those of the necessary lemmas are given in Section \ref{Sec3}. Finally, the paper concludes with a short discussion in Section \ref{Sec4}.

\section{Main results}\label{Sec2}

Before stating our results, we introduce the necessary notation and definitions. For any positive integer $k$ and real $\nu$, $\nu>-1$, we define the Rayleigh function or Rayleigh sum via
\[
\sigma _k (\nu ) = \sum\limits_{n = 1}^\infty \frac{1}{j_{\nu ,n}^{2k} } ,
\]
where $j_{\nu,n}$ denotes the $n$th positive zero of the Bessel function of the first kind $J_\nu$. The following preliminary results are important ingredients in the proof of our main results concerning the asymptotic expansions for the radii of starlikeness of normalised Bessel functions, however because of the importance of these Rayleigh sums in problems related to Bessel functions they may be of independent interest. The proofs of these lemmas are given in Section \ref{Sec3}. Throughout this paper, if not stated otherwise, empty sums are taken to be zero.

\begin{lemma}\label{lem1}
For any positive integer $k$ and real $\nu$, $\nu>k$, the Rayleigh function $\sigma_k(\nu)$ has the convergent Laurent
expansion
\begin{equation}\label{eq1}
\sigma_k(\nu)= \frac{1}{\nu ^{2k - 1}}\sum\limits_{m = 0}^\infty \frac{\sigma _m^{(k)}}{\nu ^m },
\end{equation}
where, for any fixed non-negative integer $m$, the coefficients $\sigma _m^{(k)}$ can be computed by the recurrence relation
\begin{equation}\label{eq2}
\sigma _m^{(1)}  = \frac{( - 1)^m }{4},\quad \sigma _m^{(k)}  = \sum\limits_{i = 0}^m \sum\limits_{j = 0}^i \sum\limits_{n = 1}^{k - 1} ( - k)^{m - i} \sigma _j^{(n)} \sigma _{i - j}^{(k - n)}, \quad k\geq2.
\end{equation}
In particular, we have
\[
\sigma_m^{(2)}=\frac{(-1)^m}{16}\left(2^{m+2}-m-3\right),\quad \sigma_m^{(3)}= \frac{(-1)^m}{256}\left(3^{m+4}-2^{m+7}+2(m+4)(m+6)+7\right).
\]
\end{lemma}

\begin{lemma}\label{lem2}
For any positive integer $k$ and positive real $\nu$, the Rayleigh sum $\sigma_k(\nu)$ satisfies the inequalities
\begin{equation}\label{eq4R}
\sigma _k (\nu ) \leq \binom{2k}{k}\frac{1}{2^{2k + 1} (2k - 1)}\frac{1}{\nu ^{2k - 1}}   \leq \frac{1}{\nu ^{2k - 1} }.
\end{equation}
\end{lemma}

The coefficients in our asymptotic expansions for the radii of starlikeness will be given in terms of recurrence relations involving ordinary potential polynomials of the $\sigma _m^{(k)}$'s. For any fixed complex $\mu$ and arbitrary non-negative integer $n$, the ordinary potential polynomial $\mathsf{A}_{\mu,n}$ is a polynomial in $n$ (complex) variables and is defined by the generating function
\[
\left( 1 + \sum\limits_{n = 1}^\infty x_n z^n \right)^\mu = \sum\limits_{n = 0}^\infty \mathsf{A}_{\mu,n} (x_1 ,x_2 ,\ldots ,x_n )z^n .
\]
Thus, in particular, $\mathsf{A}_{\mu,0}  = 1$, $\mathsf{A}_{\mu,1}  = \mu x_1$ and $\mathsf{A}_{\mu,2}  = \mu x_2  + \binom{\mu}{2}x_1^2$. For basic properties of these polynomials, see, e.g., \cite[Appendix]{nemes2}.

Let $\mathbb{D}_r=\{z\in\mathbb{C}:|z|<r\}$, where $r>0$, and let $f:\mathbb{D}_r\to\mathbb{C}$ be a normalised univalent or one-to-one function, which satisfies the conditions $f(0)=0$ and $f'(0)=1$, that is, $f$ is of the form $f(z)=z+a_2z^2+a_3z^3+\dots$, where the coefficients $a_2, a_3,\dots$ are real or complex numbers. The radius of univalence of the function $f$ is the largest radius $r$ for which $f$ maps univalently the open disk $\mathbb{D}_r$ into some domain in the complex plane. Similarly, the radius of starlikeness of the function $f$ is the largest radius $r$ for which $f$ maps $\mathbb{D}_r$ into a starlike domain with respect to the origin. Since the class of normalised starlike functions (with respect to the origin) is a subclass of univalent functions, the radius of univalence of $f$ is clearly allways greater or equal than the radius of starlikeness of the same function $f$. In view of the analytic characterization
of starlike functions, the radius of starlikeness is given by
\[
r^{\star}(f)=\sup\left\{r>0: \Re\! \left(\frac{zf'(z)}{f(z)}\right)>0\ \mbox{for all}\ z\in\mathbb{D}_r\right\}
\]
(see, e.g., \cite{nevanlinna}).

Now, consider the normalised Bessel function of the first kind $\phi_{\nu}:\mathbb{D}_r\to\mathbb{C}$, defined by
\begin{equation}\label{phidef}
\phi_{\nu}(z)=2^{\nu}\Gamma(\nu+1)z^{1-\frac{\nu}{2}}J_{\nu}(\sqrt{z}),
\end{equation}
where $\nu>-1$ and $J_{\nu}$ is the Bessel function of the first kind of order $\nu$ \cite[\href{http://dlmf.nist.gov/10.2.ii}{\S10.2(ii)}]{dlmf}. The function $\phi_{\nu}$ is also well-defined if $\nu<-1$ and $\nu$ is not a negative integer, however the condition $\nu>-1$ is important in our paper since only in this case the zeros of the Bessel function $J_{\nu}$ are all real (according to an old result of von Lommel \cite[Ch. XV, \S15.25]{watson}) and the reality of the zeros is essential in this paper. This is because all the results in the paper \cite{bks} on the radii of starlikeness, which we shall use here, hold under the condition that $\nu>-1$. It was shown (see \cite[Corollary 1]{bks}) that for $\nu>-1$ the radius of starlikeness of $\phi_{\nu}$ is the smallest positive root of the equation $\sqrt{z}J_\nu'(\sqrt{z})+(2-\nu)J_{\nu}(\sqrt{z})=0$. We also know that the radius of starlikeness of $\phi_{\nu}$, denoted by $r^\star(\phi_\nu)$, corresponds to the radius of univalence of $\phi_{\nu}$ (see \cite[Theorem 2]{aby}), and that it satisfies
\[
\mathop {\lim }\limits_{\nu  \to  + \infty }\frac{r^\star(\phi_\nu)}{4(\nu+1)} = 1.
\]
Moreover, by using some Euler--Rayleigh inequalities, it was proved (see also \cite[Theorem 2]{aby}) that the asymptotic formula
\begin{equation}\label{asym1}
r^\star(\phi_\nu)=4(\nu+1)\left(1-\frac{1}{\nu}+\mathcal{O}\!\left(\frac{1}{\nu^2}\right)\right)
\end{equation}
holds as $\nu\to+\infty$. The following theorem gives a complete asymptotic expansion for the radii of starlikeness of the function $\phi_\nu$.

\begin{theorem}\label{th1}
The radius of starlikeness $r^\star(\phi_\nu)$ has the asymptotic expansion
\begin{equation}\label{asymp1}
r^\star(\phi_\nu) \sim 4\nu\left(1+ \sum\limits_{k = 1}^\infty \frac{\varepsilon _k}{\nu ^k }\right) = 4\nu\left(1  + \frac{1}{\nu^2} - \frac{4}{\nu ^4} + \frac{2}{\nu ^5} + \frac{44}{\nu ^6} - \cdots\right),
\end{equation}
as $\nu \to +\infty$, where the coefficients $\varepsilon_k$ can be determined recursively by
\begin{equation}\label{eq7}
\varepsilon _k  = ( - 1)^{k + 1}  - \sum\limits_{n = 1}^{k - 1} ( - 1)^{k - n} \varepsilon _n  - \sum\limits_{n = 2}^{k + 1} 4^n \sum\limits_{m = 0}^{k - n + 1} \sigma _{k - n - m + 1}^{(n)} \mathsf{A}_{n,m} (\varepsilon _1 ,\varepsilon _2 , \ldots ,\varepsilon _m ).
\end{equation}
\end{theorem}

\begin{remark} It is readily seen from \eqref{eq2}, by using an induction argument in $k$, that $4^k \sigma_m^{(k)}$ is an integer for any $k\geq 1$ and $m\geq 0$. Consequently, the coefficients $\varepsilon_k$ are all integer valued.
\end{remark}

\begin{remark} By employing a simple algebraic manipulation, the asymptotic expansion \eqref{asymp1} may be re-arranged to the alternative form
\[
r^\star(\phi_\nu) \sim 4(\nu+1)\left(1+ \sum\limits_{k = 1}^\infty \frac{\delta_k}{\nu ^k }\right) = 4(\nu+1)\left(1 -\frac{1}{\nu} + \frac{2}{\nu^2} - \frac{2}{\nu ^3}- \frac{2}{\nu ^4} +  \cdots\right),
\]
which shares a closer resemblance to \eqref{asym1}. The coefficients $\delta_k$ can be computed via the simple recursion $\delta_1 = -1$ and $\delta_k  = \varepsilon_k - \delta_{k-1}$ for $k\geq 2$.
\end{remark}

Next, consider the normalised Bessel function of the first kind $\varphi_{\nu}:\mathbb{D}_r\to\mathbb{C}$, defined by
\begin{equation}\label{varphidef}
\varphi_{\nu}(z)=2^{\nu}\Gamma(\nu+1)z^{1-{\nu}}J_{\nu}(z),
\end{equation}
where $\nu>-1$. This function is also well-defined if $\nu<-1$ and $\nu$ is not a negative integer, however the condition $\nu>-1$ is again essential because of the reality of the zeros of $J_{\nu}$. It is known (cf. \cite[Theorem 3]{brown} or \cite[Corollary 1]{bks}) that for $\nu>-1$ the radius of starlikeness of the function $\varphi_{\nu}$ is the smallest positive root of the equation
$zJ_{\nu}'(z)+(1-\nu)J_{\nu}(z)=0$. We also know that the radius of starlikeness of $\varphi_{\nu}$, denoted by $r^\star(\varphi_\nu)$, corresponds to the radius of univalence of $\varphi_{\nu}$ (see \cite[Theorem 2]{wilf}), and, as $\nu\to+\infty$, it satisfies \cite[Theorem 2]{wilf}
\begin{equation}\label{asym2}
r^\star(\varphi_\nu)=\sqrt{2\nu}\left(1+\frac{1}{4\nu}+\mathcal{O}\!\left(\frac{1}{\nu^2}\right)\right).
\end{equation}
While studying the asymptotic behaviour of the quantity $r^\star(\varphi_\nu)$, we observed that the recurrence relation for the asymptotic expansion coefficients is much simpler and presentable if one considers instead the square of $r^\star(\varphi_{\nu})$. Therefore, in the following theorem, we provide a complete asymptotic expansion for $(r^\star(\varphi_{\nu}))^2$.
\begin{theorem}\label{th2}
The square of the radius of starlikeness $r^\star(\varphi_{\nu})$ has the asymptotic expansion
\begin{equation}\label{asymp2}
(r^\star(\varphi_{\nu}))^2\sim 2\nu\left(1+\sum\limits_{k = 1}^\infty \frac{\rho_k}{\nu^{k}}\right)=2\nu\left(1+\frac{1}{2\nu}+\frac{1}{2\nu^2}-\frac{1}{2\nu^3}-\frac{1}{2\nu^4}+\cdots\right),
\end{equation}
as $\nu \to +\infty$, where the coefficients $\rho_k$ can be determined recursively by
\[
\rho _k  = ( - 1)^{k + 1}  - \sum\limits_{n = 1}^{k - 1} ( - 1)^{k - n} \rho _n  - \sum\limits_{n = 2}^{k + 1} 2^{n + 1} \sum\limits_{m = 0}^{k - n + 1} \sigma _{k - n - m + 1}^{(n)} \mathsf{A}_{n,m} (\rho _1 ,\rho _2 , \ldots ,\rho _m ) .
\]
\end{theorem}
A simple corollary of Theorem \ref{th2} and Exercise 8.4 of Olver \cite[p. 22]{olver2} is the following generalisation of \eqref{asym2}.
\begin{corollary} The radius of starlikeness $r^\star(\varphi_{\nu})$ has the asymptotic expansion
\[
r^\star(\varphi_{\nu}) \sim \sqrt{2\nu}\left(1+\sum\limits_{k = 1}^\infty \frac{\pi_k}{\nu ^{k}}\right)=\sqrt{2\nu}\left(1+\frac{1}{4\nu}+\frac{7}{32\nu^2}-\frac{39}{128\nu^3}-\frac{405}{2048\nu^4}+\cdots\right),
\]
as $\nu \to +\infty$, where the coefficients $\pi_k$ can be computed via the recurrence relation
\[
\pi _k  = \frac{1}{2}\rho _k  - \frac{1}{2}\sum\limits_{n = 1}^{k - 1} \pi _n \pi _{k - n} .
\]
\end{corollary}

\begin{remark} A third type of normalised Bessel function of the first kind that is studied frequently in the literature is
\[
\theta_{\nu}(z)=\left( 2^\nu  \Gamma (\nu  + 1)J_\nu  (z)\right)^{\frac{1}{\nu}},
\]
with $\nu>0$. It is known (see \cite[Theorem 2]{brown} or \cite[Corollary 1]{bks}) that the radius of starlikeness of $\theta_{\nu}$, denoted by $r^\star(\theta_{\nu})$, is the smallest positive zero of $J_\nu'$. Accordingly, by a result of Olver \cite{olver},
\[
r^\star(\theta_{\nu}) \sim \nu  - \frac{\alpha }{2^{\frac{1}{3}}}\nu^{\frac{1}{3}}  + \frac{2^{\frac{1}{3}} }{10}\left( \frac{3}{2}\alpha ^2  + \frac{1}{\alpha }\right)\frac{1}{\nu ^{\frac{1}{3}}} + \frac{1}{100}\left( \frac{\alpha ^3 }{7} - 4 + \frac{1}{\alpha ^3 } \right)\frac{1}{\nu} +  \cdots 
\]
as $\nu \to +\infty$, where $\alpha$ denotes the first negative zero of the derivative of the Airy function $\operatorname{Ai}$. Olver derived his expansions for the zeros of $J_\nu'$ through an asymptotic inversion of $J_\nu'$ in its transition region. The asymptotic nature of the expansions was not demonstrated rigorously, however, it is possible to fill this gap by the methods of the present work. The details are not considered here.
\end{remark}

\section{Proofs of the results}\label{Sec3}

In this section, we prove Lemmas \ref{lem1} and \ref{lem2}, and Theorems \ref{th1} and \ref{th2}.

\begin{proof}[Proof of Lemma \ref{lem1}]
Since $\sigma _k (\nu ) = \mathcal{O}(\nu ^{1 - 2k})$ (cf. Lemma \ref{lem2}) and $\sigma _k (\nu )$ is a meromorphic function of $\nu$ with poles at $\nu=-1,-2,\ldots,-k$, it admits a Laurent series expansion of the form \eqref{eq1}, which converges for $\nu>k$. In particular, for $k=1$,
\[
\sigma _1 (\nu ) = \frac{1}{4(\nu  + 1)} = \frac{1}{\nu }\sum\limits_{m = 0}^\infty \frac{( - 1)^m}{4}\frac{1}{\nu ^m} ,\quad \nu>1
\]
(see, for instance, \cite{kerimov} or \cite[Ch. XV, \S15.51]{watson}). Hence, $\sigma _m^{(1)}  = \frac{( - 1)^m }{4}$ for $m\geq0$. To derive the recurrence formula \eqref{eq2}, we  substitute \eqref{eq1} into the known convolution relation (see, e.g., \cite{meiman}, \cite[p. 532]{kishore} or \cite{kerimov})
\begin{equation}\label{eqrec}
\sigma _k (\nu ) = \frac{1}{\nu  + k}\sum\limits_{n = 1}^{k - 1} \sigma _n (\nu )\sigma _{k - n} (\nu ),\quad k\geq 2,
\end{equation}
employ the expansion
\begin{equation}\label{expansion}
\frac{1}{\nu  + k} = \frac{1}{\nu }\sum\limits_{m = 0}^\infty \frac{( - k)^m}{\nu ^m},\quad \nu>k,
\end{equation}
and equate like powers of $\nu$.
\end{proof}

\begin{proof}[Proof of Lemma \ref{lem2}]
We prove, by induction on $k$, that for any positive integer $k$ and positive real $\nu$, the Rayleigh function $\sigma _k (\nu)$ satisfies the first inequality in \eqref{eq4R}. The second inequality follows easily by noting that for any positive integer $k$,
\[
\binom{2k}{k}\frac{1}{ 2^{2k + 1} (2k - 1) } \le \frac{ 2^{2k} }{ k^{1/2} }\frac{1}{ 2^{2k + 1} }\frac{2}{k} = \frac{1}{ k^{3/2} } \leq 1.
\]
For $k=1$, the first inequality in \eqref{eq4R} states that
\[
\sigma _1 (\nu ) \le \frac{1}{4\nu},
\]
for any $\nu>0$. This is clearly true since $\sigma _1 (\nu ) = \frac{1}{4(\nu  + 1)}$. Let $k\geq 2$ and suppose that the first inequality in \eqref{eq4R} holds for all $\sigma_m (\nu )$ with $1\leq m\leq k-1$ and $\nu>0$. Then, by employing the convolution relation \eqref{eqrec}, we find
\begin{align*}
\sigma _k (\nu ) & \le \frac{1}{\nu  + k}\sum\limits_{n = 1}^{k - 1} \binom{2n}{n}\frac{1}{2^{2n + 1} (2n - 1)}\frac{1}{\nu ^{2n - 1} }\binom{2k-2n}{k-n}\frac{1}{2^{2k - 2n + 1} (2k - 2n - 1)}\frac{1}{\nu ^{2k - 2n - 1}}
\\ & \le \frac{1}{\nu ^{2k - 1}}\sum\limits_{n = 1}^{k - 1} \binom{2n}{n}\frac{1}{2^{2n + 1} (2n - 1)}\binom{2k-2n}{k-n}\frac{1}{2^{2k - 2n + 1} (2k - 2n - 1)}
\\ & = \frac{1}{\nu ^{2k - 1}}\binom{2k}{k}\frac{1}{2^{2k + 1} (2k - 1)}.
\end{align*}
\end{proof}

\begin{proof}[Proof of Theorem \ref{th1}] Assume that $\nu \geq 0$. From the infinite product representation of the Bessel function $J_\nu$ \cite[\href{http://dlmf.nist.gov/10.21.E15}{Eq. 10.21.15}]{dlmf} and the definition \eqref{phidef} of $\phi_\nu$ we can assert that
\[
\phi_\nu (z) = z\prod\limits_{n = 1}^\infty \left( 1-\frac{z}{j_{\nu,n}^2}\right) 
\]
for any $z \in \mathbb{C}$, where $j_{\nu,n}$ denotes the $n$th positive zero of $J_\nu$. Taking the logarithm of both sides and differentiating with respect to $z$, we deduce
\begin{equation}\label{logderphi}
\frac{\phi_{\nu}'(z)}{\phi_{\nu}(z)}=\frac{1}{z}+\sum_{n=1}^\infty\frac{1}{z-j_{\nu,n}^2}.
\end{equation}
It is known (see the proof of \cite[Theorem 1]{bks}) that the radius of starlikeness of the function $\phi_{\nu}$ equals to the first positive zero of $\phi'_{\nu}$. Accordingly, \eqref{logderphi} vanishes at $r^\star(\phi_\nu),$ that is
\begin{equation}\label{logderphi2}
\frac{1}{r^\star(\phi_\nu)}=\sum_{n=1}^\infty\frac{1}{j_{\nu,n}^2-r^\star(\phi_\nu)}.
\end{equation}
Now, in view of \eqref{asym1} let $r^\star(\phi_\nu) = 4\nu(1  + \varepsilon (\nu ))$, where $\varepsilon (\nu )=\mathcal{O}(\nu^{-2})$ as $\nu \to+\infty$. Since $j_{\nu,n}>\nu$ for $\nu \geq 0$ and $n\geq 1$ (cf. \cite[\href{http://dlmf.nist.gov/10.21.E3}{Eq. 10.21.3}]{dlmf}), we can re-arrange \eqref{logderphi2} in the form
\begin{align*}
& \frac{1}{4\nu( 1 + \varepsilon (\nu ))}  = \sum\limits_{n = 1}^\infty  \frac{1}{j_{\nu ,n}^2  - 4\nu(1  + \varepsilon (\nu ))}  = \sum\limits_{n = 1}^\infty  \frac{1}{j_{\nu ,n}^2}\frac{1}{1 - \cfrac{4\nu(1  + \varepsilon (\nu ))}{j_{\nu ,n}^2 }}
 \\ & = \sum\limits_{n = 1}^\infty  \frac{1}{j_{\nu ,n}^2 }\sum\limits_{k = 0}^\infty \frac{(4\nu(1  + \varepsilon (\nu )))^k }{j_{\nu ,n}^{2k}}  = \sum\limits_{k = 0}^\infty  (4\nu(1  + \varepsilon (\nu )))^k \sum\limits_{n = 1}^\infty \frac{1}{j_{\nu ,n}^{2k + 2} }  = \sum\limits_{k = 0}^\infty  (4\nu ( 1+ \varepsilon (\nu )))^k \sigma _{k + 1} (\nu ),
\end{align*}
or
\begin{equation}\label{eqvarimp}
1 = \sum\limits_{k = 1}^\infty \left( 1 + \varepsilon (\nu )\right)^k (4\nu )^k \sigma _k (\nu ),
\end{equation}
provided $\nu$ is sufficiently large. Now, we write
\begin{equation}\label{truncasymp}
\varepsilon (\nu ) = \sum\limits_{n = 1}^{N - 1} \frac{\varepsilon_n}{\nu^n}  + R_N (\nu),
\end{equation}
where the coefficients $\varepsilon_n$ are given by the recurrence relation \eqref{eq7}. We shall prove by induction on $N$ that $R_N (\nu ) = \mathcal{O}_N(\nu ^{-N})$ for any $N\geq 1$ as $\nu \to +\infty$. (Throughout this paper, we use subscripts in the
$\mathcal{O}$ notations to indicate the dependence of the implied constant on certain parameters.) It can be shown, using the second Rayleigh sum of the zeros of $\phi'_{\nu}$ (cf. \cite{aby}), that the second positive zero is at least of order $\nu^{3/2}$, whence \eqref{truncasymp} indeed corresponds to $r^\star(\phi_\nu)$. The statement is true when $N=1$ since $R_1 (\nu ) = \varepsilon (\nu )$. Let $N\geq 2$ and suppose that the statement holds for all $R_k (\nu )$ with $1\leq k\leq N-1$. We re-write equation \eqref{eqvarimp} in the form
\begin{equation}\label{eq6}
1 = \left( 1 + \varepsilon (\nu ) \right)(4\nu )\sigma _1 (\nu ) + \sum\limits_{k = 2}^{N-1} \left( 1 + \varepsilon (\nu ) \right)^k (4\nu )^k \sigma _k (\nu )  + \sum\limits_{k = N}^\infty \left( 1 + \varepsilon (\nu ) \right)^k (4\nu )^k \sigma _k (\nu ).
\end{equation}
We would like to simplify the right-hand side as much as possible. With the aid of Lemma \ref{lem1}, the first term can be re-expressed as
\begin{align*}
& \left( 1 + \varepsilon (\nu ) \right)(4\nu )\sigma _1 (\nu ) = \left( 1 + \sum\limits_{n = 1}^{N - 1} \frac{\varepsilon _n }{\nu ^n }  + R_N (\nu ) \right)(4\nu )\sigma _1 (\nu )
\\ & = \left( 1 + \sum\limits_{n = 1}^{N - 1} \frac{\varepsilon _n }{\nu ^n} \right)(4\nu )\sigma _1 (\nu ) + R_N (\nu )(4\nu )\sigma _1 (\nu ) \\ & = \left( 1 + \sum\limits_{n = 1}^{N - 1} \frac{\varepsilon _n }{\nu ^n} \right)\sum\limits_{n = 0}^\infty  \frac{( - 1)^n}{\nu ^n }  + R_N (\nu )(4\nu )\sigma _1 (\nu )
\\ & = 1 + \sum\limits_{k = 1}^{N - 1} \left( ( - 1)^k  + \sum\limits_{n = 1}^k ( - 1)^{k - n} \varepsilon _n \right)\frac{1}{\nu ^n }  + \mathcal{O}_N\! \left( \frac{1}{\nu ^N} \right) + R_N (\nu )(4\nu )\sigma _1 (\nu ).
\end{align*}
We expand the second term on the right-hand side of \eqref{eq6} by employing Lemma \ref{lem1}, the induction hypothesis on the remainder terms $R_1(\nu),\ldots,R_{N-1}(\nu)$ and the definition of the ordinary potential polynomials:
\begin{align*} 
& \sum\limits_{k = 2}^{N - 1} \left( 1 + \varepsilon (\nu ) \right)^k (4\nu )^k \sigma _k (\nu )  = \sum\limits_{k = 2}^{N - 1} \left( 1 + \sum\limits_{n = 1}^{N - k} \frac{\varepsilon _n }{\nu ^n }  + R_{N - k + 1} (\nu ) \right)^k \frac{4^k}{\nu^{k - 1}}\sum\limits_{m = 0}^\infty  \frac{\sigma _m^{(k)} }{\nu ^m } 
\\ & = \sum\limits_{k = 2}^{N - 1} \left( 1 + \sum\limits_{n = 1}^{N - k} \frac{\varepsilon _n }{\nu ^n }  + \mathcal{O}_{N - k + 1}\! \left( \frac{1}{\nu ^{N - k + 1} } \right) \right)^k \frac{4^k}{\nu^{k - 1}}\sum\limits_{m = 0}^\infty \frac{\sigma _m^{(k)} }{\nu ^m }
\\ &= \sum\limits_{k = 2}^{N - 1} \left( \sum\limits_{n = 0}^{N - k} \mathsf{A}_{k,n} (\varepsilon _1 ,\varepsilon _2 , \ldots ,\varepsilon _n )\frac{1}{\nu ^n }  + \mathcal{O}_{N - k + 1}\!  \left( \frac{1}{\nu ^{N - k + 1} } \right) \right)\frac{4^k }{\nu^{k - 1}}\sum\limits_{m = 0}^\infty  \frac{\sigma _m^{(k)} }{\nu ^m } 
 \\ &= \sum\limits_{k = 2}^{N - 1} \frac{4^k }{\nu^{k - 1}}\left( \sum\limits_{n = 0}^{N - k} \left( \sum\limits_{m = 0}^n \sigma _{n - m}^{(k)} \mathsf{A}_{k,m} (\varepsilon _1 ,\varepsilon _2 , \ldots ,\varepsilon _m )  \right)\frac{1}{\nu ^n}  + \mathcal{O}_{N - k + 1}\!  \left( \frac{1}{\nu ^{N - k + 1} } \right) \right)
\\ & = \sum\limits_{k = 2}^{N - 1} 4^k \sum\limits_{n = 0}^{N - k} \left( \sum\limits_{m = 0}^n \sigma _{n - m}^{(k)} \mathsf{A}_{k,m} (\varepsilon _1 ,\varepsilon _2 , \ldots ,\varepsilon _m ) \right)\frac{1}{\nu ^{n + k - 1}}  + \mathcal{O}_N\!  \left( \frac{1}{\nu ^N } \right)
\\ & = \sum\limits_{k = 1}^{N - 2} \left( \sum\limits_{n = 2}^{k + 1} 4^n \sum\limits_{m = 0}^{k - n + 1} \sigma _{k - n - m + 1}^{(n)} \mathsf{A}_{n,m} (\varepsilon _1 ,\varepsilon _2 , \ldots ,\varepsilon _m )  \right)\frac{1}{\nu ^k } 
\\ & \quad\; + \left( \sum\limits_{n = 2}^{N - 1} 4^n \sum\limits_{m = 0}^{N - n} \sigma _{N - n - m}^{(n)} \mathsf{A}_{n,m} (\varepsilon _1 ,\varepsilon _2 , \ldots ,\varepsilon _m )  \right)\frac{1}{\nu ^{N - 1} } + \mathcal{O}_N\!  \left( \frac{1}{\nu ^N }\right).
\end{align*}
Finally, from Lemmas \ref{lem1} and \ref{lem2}, we can infer that the third term on the right-hand side of \eqref{eq6} is
\begin{align*}
& \sum\limits_{k = N}^\infty  \left( 1 + \varepsilon (\nu ) \right)^k (4\nu )^k \sigma _k (\nu )  = \left( 1 + \varepsilon (\nu ) \right)^{N } (4\nu )^{N } \sigma _{N } (\nu ) + \sum\limits_{k = N + 1}^\infty  \left( 1 + \varepsilon (\nu ) \right)^k (4\nu )^k \sigma _k (\nu )
\\ & = \left( 1 + \varepsilon (\nu ) \right)^{N} (4\nu )^{N} \left( \frac{\sigma _0^{(N)} }{\nu ^{2N - 1} } + \frac{1}{\nu ^{2N - 1} }\sum\limits_{m = 1}^\infty  \frac{\sigma _m^{(N)} }{\nu ^m } \right) + \sum\limits_{k = N + 1}^\infty  \left( 1 + \varepsilon (\nu )\right)^k (4\nu )^k \sigma _k (\nu )
\\ & = \left( 1 + \varepsilon (\nu )\right)^{N } \frac{4^{N} \sigma _0^{(N)} }{\nu ^{N-1} } + \left( 1 + \varepsilon (\nu )\right)^{N}  \frac{4^{N}}{\nu ^{N-1} }\sum\limits_{m = 1}^\infty  \frac{\sigma _m^{(N)} }{\nu ^m }  + \sum\limits_{k = N + 1}^\infty  \left( 1 + \varepsilon (\nu)\right)^k (4\nu )^k \sigma _k (\nu )
\\ & = \frac{4^{N} \sigma _0^{(N)} }{\nu ^{N-1} }+ \mathcal{O}_N \! \left( \frac{1}{\nu ^{N} }\right) + \sum\limits_{k = N + 1}^\infty  \left( 1 + \varepsilon (\nu ) \right)^k (4\nu )^k \sigma _k (\nu )
\\ & = \frac{4^{N} \sigma _0^{(N)} }{\nu ^{N-1}} + \mathcal{O}_N \! \left( \frac{1}{\nu ^{N} }\right) + \mathcal{O}(1)\sum\limits_{k = N + 1}^\infty  \left( 1 + \varepsilon (\nu ) \right)^k (4\nu )^k \frac{1}{\nu ^{2k - 1} } 
\\ & = \frac{4^{N} \sigma _0^{(N)} }{\nu ^{N-1} } + \mathcal{O}_N \! \left( \frac{1}{\nu ^{N} } \right) + \mathcal{O}(1)\left( 4 + 4\varepsilon (\nu ) \right)^{N + 1} \frac{1}{\nu ^{N}}\sum\limits_{k = 0}^\infty \left( 4 + 4\varepsilon (\nu ) \right)^k \frac{1}{\nu ^k}
\\ & = \frac{4^{N } \sigma _0^{(N)}}{\nu ^{N-1}} + \mathcal{O}_N \! \left( \frac{1}{\nu ^{N} } \right) +\mathcal{O}(1) \left( 4 + o(1)\right)^{N + 1} \frac{1}{\nu ^{N} }\sum\limits_{k = 0}^\infty  \left( 4 + o(1) \right)^k \frac{1}{\nu ^k} 
\\ & = \frac{4^{N} \sigma _0^{(N)} }{\nu ^{N-1}} + \mathcal{O}_N \! \left( \frac{1}{\nu ^{N} }\right) + \mathcal{O}_N \! \left( \frac{1}{\nu ^{N} } \right) = \frac{4^{N} \sigma _0^{(N)} }{\nu ^{N-1}} + \mathcal{O}_N \! \left( \frac{1}{\nu ^{N} } \right).
\end{align*}
Substituting these results into \eqref{eq6} and applying the recurrence relation \eqref{eq7} in the form
\[
( - 1)^k  + \sum\limits_{n = 1}^k ( - 1)^{k - n} \varepsilon _n   + \sum\limits_{n = 2}^{k + 1} 4^n \sum\limits_{m = 0}^{k - n + 1} \sigma _{k - n - m + 1}^{(n)} \mathsf{A}_{n,m} (\varepsilon _1 ,\varepsilon _2 , \ldots ,\varepsilon _m )   = 0,
\]
equation \eqref{eq6} simplifies to
\[
1 = 1 + \mathcal{O}_N\! \left( \frac{1}{\nu ^{N} } \right) +  R_N (\nu )(4\nu)\sigma _1 (\nu )+ \mathcal{O}_N\! \left( \frac{1}{\nu ^{N} } \right) + \mathcal{O}_N\! \left( \frac{1}{\nu ^{N} } \right),
\]
i.e.,
\[
R_N (\nu )\sigma _1 (\nu ) =\mathcal{O}_N \!\left( \frac{1}{\nu ^{N + 1} } \right)
\]
as $\nu \to +\infty$. Since $\sigma _1 (\nu ) = \frac{1}{ 4(\nu  + 1) }$, it follows that $R_N (\nu ) = \mathcal{O}_N(\nu ^{-N})$, completing the proof of Theorem \ref{th1}.
\end{proof}

\begin{proof}[Proof of Theorem \ref{th2}] We can proceed similarly as in the case of Theorem \ref{th2}. Assume that $\nu \geq 0$. Combining the infinite product representation of the Bessel function with the definition \eqref{varphidef} of the function $\varphi_\nu$ yields
\[
\varphi_\nu (z) = z\prod\limits_{n = 1}^\infty \left( 1 - \frac{z^2 }{j_{\nu ,n}^2 } \right) 
\]
for any $z \in \mathbb{C}$. Taking the logarithm of each side and differentiating with respect to $z$, we obtain
\begin{equation}\label{logdervarphi}
\frac{\varphi_{\nu}'(z)}{\varphi_{\nu}(z)}=\frac{1}{z}-\sum_{n=1}^\infty\frac{2z}{j_{\nu,n}^2-z^2}.
\end{equation}
It is known (see \cite[Theorem 3]{brown} or the proof of \cite[Theorem 1]{bks}) that the radius of starlikeness of the function $\varphi_{\nu}$ is equal to the first positive zero of $\varphi'_{\nu}$. Consequently, \eqref{logdervarphi} vanishes at $r^\star(\varphi_\nu)$, that is
\begin{equation}\label{logdervarphi2}
\frac{1}{\left(r^\star(\varphi_\nu)\right)^2}=\sum_{n=1}^\infty\frac{2}{j_{\nu,n}^2-\left(r^\star(\varphi_\nu)\right)^2}.
\end{equation}
Now, in view of \eqref{asym2} let $r^\star(\varphi _\nu) = \sqrt {2\nu (1 + \rho (\nu ))}$, where $\rho (\nu )=\mathcal{O}(\nu^{-1})$ as $\nu \to+\infty$. With this notation, \eqref{logdervarphi2} can be re-arranged in the form
\begin{align*}
\frac{1}{2\nu (1 + \rho (\nu ))} &  = \sum\limits_{n = 1}^\infty  \frac{2}{j_{\nu ,n}^2  - 2\nu (1 + \rho (\nu ))}  = \sum\limits_{n = 1}^\infty  \frac{1}{j_{\nu ,n}^2 }\frac{2}{1 - \cfrac{2\nu (1 + \rho (\nu ))}{j_{\nu ,n}^2 }}
\\ & = 2\sum\limits_{n = 1}^\infty  \frac{1}{j_{\nu ,n}^2 }\sum\limits_{k = 0}^\infty  \frac{(2\nu (1 + \rho (\nu )))^k }{j_{\nu ,n}^{2k}} = 2\sum\limits_{k = 0}^\infty  (2\nu (1 + \rho (\nu )))^k \sum\limits_{n = 1}^\infty \frac{1}{j_{\nu ,n}^{2k + 2} } \\ & = 2\sum\limits_{k = 0}^\infty  (2\nu (1 + \rho (\nu )))^k \sigma _{k + 1} (\nu ),
\end{align*}
or
\[
\frac{1}{2} = \sum\limits_{k = 1}^\infty (1 + \rho (\nu ))^k (2\nu )^k \sigma _k (\nu ) ,
\]
provided $\nu$ is sufficiently large. The remaining part of the proof is completely analogous to that of Theorem \ref{th1}, and we leave the details to the interested reader.
\end{proof}

\section{Discussion and future work}\label{Sec4}

In this paper, by using convergent Laurent series expansions for the Rayleigh sums of the positive zeros of the Bessel function of the first kind and asymptotic inversion, we obtained complete asymptotic expansions for the radii of starlikeness of two types of normalised Bessel functions. An essential ingredient of the proofs was the fact that the radii of starlikeness of these two normalised Bessel functions are equal to the radii of univalence of the same functions, which are in turn the first positive zeros of their respective derivatives.

It is expected that the techniques employed in the paper could be useful to treat similar problems related to other special functions, such as the Kummer hypergeometric functions, the Coulomb wave functions or the $q$-Bessel functions. In the case of the Kummer hypergeometric functions, the classical approach to consider the range of the parameters when all the zeros are real does not work, however we believe that the radii of starlikeness of normalised Kummer functions can be studied through continued fractions. Regarding the Coulomb wave functions and the $q$-Bessel functions, we expect that our approach can be used directly. A possible difficulty could arise from the more complicated nature of the corresponding Rayleigh sums of the positive (and negative) zeros of these functions (cf. \cite{bmee} and \cite{ab}).

Another direction in research may be the study of the radii of convexity of the normalised Bessel-, Kummer hypergeometric-, Coulomb wave- and $q$-Bessel functions and the derivation of asymptotic expansions for these quantities when the parameters become large. Apparently, the methods of the present paper are not directly applicable in these cases, however, we believe that by an appropriate reformulation of the problems or modification of our techniques, they can be tackled successfully.

\section*{Acknowledgement} The second author was supported by a Premium Postdoctoral Fellowship of the Hungarian Academy of Sciences. The authors thank the referee for helpful comments and suggestions for improving the presentation.

\end{document}